\documentclass{article}

\usepackage{amsthm}
\usepackage{epsfig}
\usepackage{amssymb}

\newtheorem{thm}{Theorem}
\newtheorem{lemma}{Lemma}

\def\u{{\bf u}}
\def\w{{\bf w}}
\def\span{{\mbox{span}}}
\def\v{{\bf v}}
\def\a{{\alpha}}
\def\b{{\beta}}
\def\c{{\gamma}}
\def\d{{\delta}}
\def\+{{\cup\,}}
\def\f{{\Psi}}
\def\x{{\vartheta}}
\def\Re{{\mathbb{R}}}
\def\dc{{\overline{d}}}

\title{Universal spectra of the disjoint union of regular graphs}
\author{Willem H. Haemers\thanks{e-mail haemers@uvt.nl}
\\
{\it\small Department of Econometrics and Operations Research,}
\\
{\it\small Tilburg University, Tilburg, The Netherlands}
\\[8pt]
Mohammad Reza Oboudi\thanks{e-mail mr\_oboudi@yahoo.com}
\\
{\it\small Department of Mathematics, College of Sciences,}
\\
{\it\small Shiraz University, Shiraz, 71457-44776, Iran}
}
\date{}

\begin{document}
	
\maketitle
\begin{abstract}
A universal adjacency matrix of a graph $G$ with adjacency matrix $A$
is any matrix of the form $U = \a A + \b I + \c J + \d D$ with $\a \neq 0$,
where $I$ is the identity matrix, $J$ is the all-ones matrix and $D$ is the diagonal matrix with the vertex degrees.
In the case that $G$ is the disjoint union of regular graphs,
we present an expression for the characteristic polynomials of the various universal adjacency matrices
in terms of the characteristic polynomials of the adjacency matrices of the components.
As a consequence we obtain a formula for the characteristic polynomial of the Seidel matrix of $G$,
and the signless Laplacian of the complement of $G$ (i.e. the join of regular graphs).
\\[3pt]
{\bf Keywords}: Graph spectrum, characteristic polynomial, universal adjacency matrix, Seidel matrix, Laplacian, signless Laplacian.
\\
{\bf AMS subject classification}: 05C50.
\end{abstract}

\section{Introduction}

Throughout $G$ is a graph of order $n$ with adjacency matrix $A$.
Let $I$, $J$ and $D$ be the identity matrix, the all-ones matrix and the diagonal matrix with the degrees of $G$ on the diagonal,
respectively.
Any matrix of the form $U=\a A+\b I + \c J + \d D$ with $\a,\b,\c,\d\in\Re$, $\a\neq 0$ (as usuall, $\Re$ is the set of real numbers)
is called a {\em universal adjacency matrix} of $G$.
Several well-studied types of matrices associated to $G$ are a special case of a universal matrix.
For example, if $(\alpha,\beta,\c,\delta)=(1,0,0,0)$, $(-1,0,0,1)$, $(1,0,0,1)$, $(-2,-1,1,0)$ we obtain the adjacency,
the Laplacian, the signless Laplacian and the Seidel matrix, respectively,
and for $(\a,\b,\c,\d)=(-1,-1,1,0)$, $(1,n,-1,-1)$, $(-1,n-2,1,-1)$, $(2,1,-1,0)$
we get the matrices of the complement $\overline{G}$ of $G$ of the mentioned types.

If $G$ is regular of degree $d$, then there is an easy relation between the spectra of the various universal adjacency matrices.
Indeed, if $d=\x_1,\ \x_2,\ \ldots,\ \x_n$ are the eigenvalues of $A$, then
$\a d +\b+\c n+\d d,\ \a \x_2+\b+\d d,\ \ldots,\ \a \x_n+\b+\d d$ are the eigenvalues of $U$.
Let us denote the characteristic polynomial of a matrix $M$, $\det(\x I-M)$, by $\f(M,\x)$.
Then the relation between the spectrum of $A$ and the spectrum of $U$ can be formulated as:
\begin{equation}\label{UA}
\frac{\f(U,\a \x+\b+\d d)}{\a\x-\a d - \c n}
= \a^{n-1} \frac{\f(A,\x)}{\x-d}.
\end{equation}
In this note we consider the case that 
$G$ is the disjoint union of $m$ regular graphs $G_1,\ldots,G_m$.
We obtain an expression for the characteristic polynomials of a universal adjacency matrix of $G$
in terms of the characteristic polynomial of the adjacency matrices of $G_1,\ldots,G_m$.
The complement $\overline{G}$ of $G$ is known as the {\em join} of the complements of $G_1,\ldots,G_m$.
Since the class of universal adjacency matrices of a graph $G$ is closed under taking complements,
our result also applies to the join of regular graphs.
If $\c=0$, $U$ is a block diagonal matrix and the result is trivial.
Also some cases with $\c\neq 0$ are known.
For the adjacency matrix of $\overline{G}$ and for the Laplacian of $\overline{G}$ we refer to~\cite{CFMR,S}.
However, we haven't seen a common generalization of these results, and
for the Seidel matrix of $G$ and the signless Laplacian of $\overline{G}$ the formula seems to be new.
%
%
The result follows rather easily from a known lemma on the eigenvalues of an equitable partition of a symmetric matrix.
Part of the motivation for writing this note is to illustrate the use  of this technique.

\section{Equitable partitions}\label{ep}

Consider a symmetric matrix $U$ over $\Re$ with rows and columns indexed by $X=\{1,\ldots,n\}$.
Let $\{X_1,\ldots,X_m\}$ be a partition of $X$ with characteristic vectors $\v_1,\ldots,\v_m$,
and define ${\cal V}=\span\{\v_1,\ldots,\v_m\}$ over $\Re$.
Let $U$ be partitioned accordingly:
\[
U=
\left[
\begin{array}{ccc}
U_{1,1} & \cdots & U_{1,m}\\
\vdots & & \vdots\\
U_{m,1} & \cdots & U_{m,m}
\end{array}
\right].
\]
Assume the partition is {\em equitable} with {\em quotient matrix} $Q$,
which means that each block $U_{i,j}$ has constant row sum equal to $(Q)_{i,j}$ ($1\leq i,j\leq m$).
Then
(see \cite{BH}, Section 2.3):

\begin{lemma}\label{main}
The spectrum of $U$ consists of two types of eigenvalues:\\
Type 1: the eigenvalues with eigenvectors in $\cal V$, which coincide with the eigenvalues of $Q$.\\
Type 2: the eigenvalues with eigenvectors in ${\cal V}^\perp$, the orthogonal complement of $\cal V$;
these eigenvalues coincide with the type~2 eigenvalues of any matrix $U^*$ obtained from $U$ by adding $\c_{i,j}J$ to $U_{i,j}$
for constants $\c_{i,j}$.
\end{lemma}
\begin{proof}
Consider the matrix $V=[\v_1,\ldots,\v_m]$ (i.e. $(V)_{i,j}=1$ if $i\in X_j$ and $0$ otherwise).
Since the partition is equitable it follows that $UV=VQ$.
If $\lambda$ is an eigenvalue of $Q$ with eigenvector $\u$, then
$UV\u=VQ\u=\lambda V\u$.
Therefore $\lambda$ is an eigenvalue of $U$ with eigenvector $V\u\in{\cal V}$.
Thus we obtain $m$ eigenvalues of $U$ with $m$ independent eigenvectors in ${\cal V}$.
Since $U$ is symmetric, every other eigenvector is in ${\cal V}^\perp$.
Let $\w\in{\cal V}^\perp$ be an eigenvector of $U$ with eigenvalue $\mu$, then $\mu\w = U\w = U^* \w$.
\end{proof}

\section{Application}

Let us apply Lemma~\ref{main} to the universal adjacency matrix $U=\a A+\b I +\c J +\d D$
of $G$ when $G$ is the disjoint union 
of $m$ regular graphs $G_1,\ldots,G_m$.
Suppose $G_i$ has adjacency matrix $A_i$, order $n_i$ and degree $d_i$.
Define $c_i=(\a+\d)d_i+\b$.
Then $U$ has an obvious partition with quotient matrix:
\[
Q=\left[
\begin{array}{cccc}
c_1+\c n_1 & \c n_2 & \cdots & \c n_m\\
\c n_1 & c_2+\c n_2 & \cdots & \c n_m\\
\vdots & \vdots & \ddots & \vdots\\
\c n_1 & \c n_2 & \cdots & c_m+\c n_m
\end{array}
\right].
\]
We easily find the characteristic polynomial $\f(Q,\x)=\det(\x I-Q)$ of $Q$ by Gaussian elimination:
first we subtract the first row of $\x I-Q$ from every other row and then we evaluate with respect to the first row.
Thus we find:
\[
\f(Q,\x) = \left(1-\sum_{i=1}^m \frac{\c n_i}{\x-c_i}\right) \prod_{i=1}^m (\x-c_i).
\]
Define $U^*=\a A +\b I + \d D = U-\c J$, and let $Q^*$ be the quotient matrix of $U^*$.
From Lemma~\ref{main} it follows that $U$ and $U^*$ have the same type~2 eigenvalues.
The type~1 eigenvalues of $Q$ and $Q^*$ are the roots of $\f(Q,\x)$ and $\f(Q^*,\x)$, respectively.
The lemma gives that $\f(Q,\x)$ divides $\f(U,\x)$, that $\f(Q^*,\x)$ divides $\f(U^*,\x)$,
and that
\begin{equation}\label{U/Q}
\frac{\f(U,\x)}{\f(Q,\x)} = \frac{\f(U^*,\x)}{\f(Q^*,\x)}.
\end{equation}
Indeed, the roots of these polynomials are the type~2 eigenvalues of $U$ and $U^*$, respectively.
Clearly $U^*$ is the block diagonal matrix with diagonal blocks $U_i^*=\alpha A_i + (\b + \d d_i) I$ ($i=1,\ldots,m$),
and $Q^*$ is the diagonal matrix with diagonal entries $c_1,\ldots,c_m$.
With the help of Equation~\ref{UA} we get
\[
\f(U^*,\x) = \prod_{i=1}^m \f(U^*_i,\x) = \prod_{i=1}^m \alpha^{n_i} \f(A_i,(\x - \beta - \delta d_i)/\alpha) \mbox{ and }
\f(Q^*,\x) = \prod_{i=1}^m (\x-c_i).
\]
By use of Equation~\ref{U/Q} and the above equations we obtain:

\begin{thm}\label{U}
Suppose $U=\a A+\b I+\c J +\d D$ is a universal adjacency matrix of a graph~$G$,
which is the disjoint union of $m$ regular graphs $G_1,\ldots,G_m$.
Let $G_i$ have order $n_i$, degree $d_i$ and adjacency matrix $A_i$ ($i=1,\ldots,m$).
Then the characteristic polynomial of $U$ is equal~to:
\[
\f(U,\x) = 
\left( 1-\sum_{i=1}^m\frac{\c n_i}{\x-(\a+\d)d_i-\b} \right)
\prod_{i=1}^m \alpha^{n_i} \f(A_i,(\x - \beta - \delta d_i)/\alpha).
\]
\end{thm}

\section{Examples}

We consider two corollaries of this theorem, which we believe to be new.
If $(\a,\b,\c,\d)=(-2,-1,1,0)$ then $U$ is the Seidel matrix $S$ of $G$, and thus we have
the following formula
\[
\f(S,\x) =
\left( 1-\sum_{i=1}^m\frac{n_i}{\x+2d_i+1} \right)
\prod_{i=1}^m (-2)^{n_i} \f(A_i,(\x + 1)/{-2}).
\]
In terms the characteristic polynomials of the Seidel matrices $S_i$ of $G_i$ this becomes:
\[
\f(S,\x) =
\left( 1-\sum_{i=1}^m\frac{n_i}{\x+2d_i+1} \right)
\prod_{i=1}^m \frac{\x+2d_i+1}{\x-n_i+2d_i+1}\f(S_i,\x).
\]
For the second example we recall that $\overline{G}$ is the join of $\overline{G}_1,\ldots,\overline{G}_m$.
Let $T$ and $T_1,\ldots,T_m$ be the signless Laplacians of $\overline{G}$ and
$\overline{G}_1,\ldots,\overline{G}_m$, respectively.
Then $T$ is a universal adjacency matrix of $G$ with $(\a,\b,\c,\d)=(-1,n-2,1,-1)$.
Theorem~\ref{main} gives:
\[
\f(T,\x) =
\left( 1-\sum_{i=1}^m\frac{n_i}{\x-n+2d_i+2} \right)
(-1)^n \prod_{i=1}^m \f(A_i,-\x+n-d_i-2).
\]
By use of Equation~\ref{UA} we find an expression for $\f(T,\x)$ in terms of the spectra of $T_i$:
\[
\f(T,\x) = \left( 1-\sum_{i=1}^m\frac{n_i}{\x-n+2n_i-2\dc_i} \right)
\prod_{i=1}^m \frac{\x-n+2n_i-2\dc_i}{\x-n+n_i-2\dc_i} \f(T_i,\x-n+n_i),
\]
where $\dc_i=n_i-d_i-1$ is the degree of $\overline{G}_i$.
If $m=2$ this leads to a known formula for the signless Laplacian of the join of two regular graphs (see~\cite{LL}, Corollary~2.4):
\[
\f(T,\x) = \left(1-\frac{n_1 n_2}{(\x-n_1-2\dc_2)(\x-n_2-2\dc_1)} \right)\f(T_1,\x-n_2)\f(T_2,\x-n_1).
\]

\section{Concluding remark}
Although the spectrum of every possible universal adjacency matrix of the disjoint union of regular graphs follows from Theorem~\ref{U},
it still may be a bit tricky to work it out in a specific case.
Sometimes it is easier to work directly from Lemma~\ref{main}.
For example when $S = -2A+J-I$ is the Seidel matrix of a complete multipartite graph.
Then all type~2 eigenvalues of $S$ are equal to $-1$
(indeed, we obtain $-I$ if we subtract $J$ from the diagonal blocks and add $J$ to the other blocks).
The type~1 eigenvalues of $S$ are the eigenvalues of the quotient matrix which has characteristic polynomial
$
\left(1+\sum_{i=1}^{m} n_i/(\x-2n_i+1)\right)\prod_{i=1}^m (\x-2n_i+1)
$
This result has been obtained in~\cite{WZL} by a different approach.


\begin{thebibliography}{99}

\bibitem{BH}
A.E. Brouwer and W.H. Haemers, {\lq}Spectra of Graphs{\rq}, Springer, 2012.

\bibitem{CFMR}
D.M. Cardoso, M.A.A. Freitas, E.A. Martins, M. Robbiano,
Spectra of graphs obtained by a generalization of the join graph operation,
Discrete Math. {\bf 313} (2013): 733-741.

\bibitem{LL}
X. Liu, P. Lu,
Signless Laplacian spectral characterization of some joins,
Electronic Journal of Linear Algebra {\bf 30} (2015) 443-454.

\bibitem{S}
A.J. Schwenk, Computing the characteristic polynoimial of a graph, in:
R. Bary, F. Harary(Eds.), Graphs Combinatorics, in: Lecture Notes in
Mathematics, vol. 406, Springer, Berlin (1974): 153--172.

\bibitem{WZL}
L. Wang, G. Zhao, K. Li,
Seidel integral Complete r-Partite Graphs,
Graphs and Combinatorics {\bf 30} (2014) 479-493.

\end{thebibliography}
\end{document}